\documentclass[nolineno]{llncs}

\usepackage[utf8]{inputenc}
\usepackage[english]{babel}
\usepackage{amssymb,amsmath}
\usepackage{comment}
\usepackage{enumitem}
\usepackage{thm-restate}

\usepackage{physics}
\usepackage{tikz}
\usepackage{mathdots}
\usepackage{yhmath}
\usepackage{cancel}
\usepackage{color}
\usepackage{siunitx}
\usepackage{array}
\usepackage{multirow}
\usepackage{gensymb}
\usepackage{tabularx}
\usepackage{extarrows}
\usepackage{booktabs}
\usepackage[disable]{todonotes}
\usepackage{xcolor}

\usepackage{academicons}
\definecolor{orcidlogocol}{HTML}{A6CE39}

\usetikzlibrary{fadings}
\usetikzlibrary{patterns}
\usetikzlibrary{shadows.blur}
\usetikzlibrary{shapes}
\usepackage{xspace}

\usepackage{mathtools} 

\usepackage{graphicx}
\usepackage[subrefformat=parens,labelfont=default,font=small]{subcaption}
\graphicspath{{figures/}}

\newcommand{\intr}[1]{\mbox{int}({#1})}
\newcommand{\extr}[1]{\mbox{ext}({#1})}
\newcommand{\sk}[1]{\todo[color=cyan!10]{#1}\xspace}

\newcommand{\shahin}[1]{\todo[color=cyan!10]{#1}\xspace}
\newcommand{\shahininline}[1]{\todo[inline,color=cyan!10]{#1}\xspace}
\newcommand{\avery}[1]{\todo[color=green!10]{#1}\xspace}
\newcommand{\averyinline}[1]

\title{Cops and Robbers on 1-Planar Graphs\thanks{This research is funded in part by the Natural Sciences and Engineering Research Council of Canada (NSERC).}
}

\author{Stephane~Durocher\inst{1}
 \and Shahin~Kamali\inst{2}
 \and Myroslav~Kryven\inst{1}\orcidID{0000-0003-4778-3703}
 \and Fengyi~Liu\inst{1}
 \and Amirhossein~Mashghdoust\inst{1}
 \and Avery~Miller\inst{1}
 \and Pouria~Zamani~Nezhad\inst{1}
 \and Ikaro~Penha~Costa\inst{1}
 \and Timothy~Zapp\inst{1}
}

 \institute{Dept.~Computer Science, University of Manitoba, Winnipeg, Manitoba, Canada\\
   \texttt{firstname.lastname@umanitoba.ca}, \{\texttt{liuf3412}, \texttt{mashghda}, \texttt{zamaninp}, \texttt{zappt3}\}\texttt{@myumanitoba.ca} 
 \and Dept.~Electrical Engineering and Computer Science, York University, Toronto, Ontario, Canada\\
 \texttt{kamalis@yorku.ca}}

\begin{document}

\maketitle

\begin{abstract}
\emph{Cops and Robbers} is a well-studied pursuit-evasion game in which a set of cops seeks to catch a robber in a graph $G$, where cops and robber move along edges of $G$. The \emph{cop number} of $G$ is the minimum number of cops that is sufficient to catch the robber. Every planar graph has cop number at most three, and there are planar graphs for which three cops are necessary [Aigner and Fromme, DAM 1984]. 
We study the problem for 
\emph{1-planar graphs}, that is, graphs that can be drawn in the plane with at most one crossing per edge.
In contrast to planar graphs, we show that some 1-planar graphs have unbounded cop number. Meanwhile, for maximal 1-planar graphs, we prove that three cops are always sufficient and sometimes necessary. In addition, we completely determine the cop number of outer 1-planar graphs.

\end{abstract}

\section{Introduction}\label{sec:intro}
\emph{Pursuit–evasion} is a family of problems (also called games) in which one group seeks to capture members of another group in a given environment. There are many variants of pursuit–evasion games, depending on the game environment, information available to each player about the environment and other players, and restrictions on the freedom or speed at which players can move.

One of the most common and well-studied pursuit-evasion problems is the game of \emph{Cops and Robbers} on graphs, which was
formalized by Quilliot~\cite{quilliot-french-1978} and Nowakowski and Winkler~\cite{Nowakowski1983VertextovertexPI} in the 1980s; see also the recent book by Bonato and Nowakowski~\cite{cr-book}.
The game is played in a graph by two players: the robber player and the cop player. The common assumption that we also adopt here is that each player has full information about the graph and the other player's moves. The game consists of rounds (or steps) on a given graph. In the initial round, the cop player selects starting vertices for a set of cops, and then the robber player selects a starting vertex for a robber. In the subsequent rounds, the players alternate turns; during the cops' turn, the cop player may move some of the cops to adjacent vertices. Similarly, the robber player may move the robber to an adjacent vertex during the robber's turn. 
 The cop player wins if the robber and any of the cops are simultaneously on the same vertex; otherwise, when the game continues indefinitely, the robber player wins. If a single cop suffices to catch the robber in a graph $G$, even when the robber plays adversarially,\shahin{is ``adversarial" is the right word?, PZ: Do you think we should use Optimally? SD: adversarially is more general than optimally. Let's keep adversarially.}
 then $G$ is a \emph{cop-win graph}; otherwise, $G$ is a \emph{robber-win graph}. 
 The minimum number of cops necessary to catch the robber in $G$, denoted $c(G)$, is called the \emph{cop number} of $G$, and $G$ is a \emph{$c(G)$-cop win graph}.

\paragraph{\textbf{Related work}}
Various characterizations of $c$-cop win graphs are known~\cite{clarke-characterizations-2012}. For example, every cop-win graph has a \emph{domination elimination ordering sequence}~\cite{quilliot-french-1978} (also known as \emph{dismantling ordering}), i.e., an ordering of the vertices such that, if the vertices are removed in this order, each vertex (except the last) is dominated at the time it is removed. This characterization helped identify classes of cop-win graphs, such as chordal graphs (as every elimination ordering of a chordal graph is also a dismantling ordering) and visibility graphs~\cite{lubiw-visibility-graphs-dismantable-2017}. 

For some classes of graphs, a cop strategy comes directly from the graph's structural properties.
For example,
Aigner and Fromme~\cite{af-agocar-82} showed that every planar graph has a cop number of at most three by implicitly using 
the Jordan curve property, and, the fact that a cycle in a \emph{plane graph} (a graph drawn in the plane without edge crossings) partitions the graph into interior and exterior regions. Some classes of planar graphs have cop number two (e.g., series-parallel graphs~\cite{subdivision-lemma-and-tw-over-2} and outerplanar graphs).
 On the other hand, planar graphs with cop number three are known, e.g., the \emph{dodecahedron}, which is the skeleton of the platonic solid with faces of degree five. It is conjectured that the dodecahedron is the smallest planar graph with cop number three~\cite{af-agocar-82}. 
 Maurer et al.~\cite{workshop} give an example of a \emph{triangulation} (i.e., maximal planar graph) with cop number three.

Aigner and Fromme~\cite{af-agocar-82} also asked whether the result concerning planar graphs can be generalized to graphs of higher \emph{orientable genus}, which is the minimum number of handles attached to a sphere so that the graph could be drawn without crossings on the resulting surface. Shortly after, Quilliot \cite{q-asnapgpoagwagg-83} gave an upper bound of $2g+3$ on the cop number of graphs of genus $g$. Schr\"oder~\cite{schroder-3o2g-2001} improved this bound to $1.5g+3$ and conjectured that there is a tighter upper bound of $g+3$.  The current best upper bound of approximately $1.268g$ is due to Erde and Lehner~\cite{el-ibotcnoagdoas-21}; find details in the survey of Bonato and Mohar~\cite{bm-tdicar-17}. \shahininline{Do ``genus" and ``orientable genus" refer to the same thing? MK: the question as posed by Aigner and Fromme~\cite{af-agocar-82} was concerning only drawings on orientable surfaces,  by specifying ``orientable genus" we highlight that}
%
%

Having established the cop number of planar graphs, a natural next question concerns the cop number of graphs that are \emph{almost planar}.
Proximity to being planar for a given graph $G$ can be parameterized by the minimum number of crossings per edge among all drawings of $G$ in the plane, where a planar graph requires zero crossings per edge. \shahininline{Technically, you minimum maximum number of crossing, where minimum is taken over all drawings and max is taken over all edges. SD: True, but the sense comes across in the current wording, and I don't have a better way to express this concisely.}
In recent years, there has been an increasing interest in the family of graphs that generalize planar graphs, called \emph{beyond-planar graphs}, that is, graphs that can be drawn with few crossings~\cite{bp-book}. 
A common subclass of beyond-planar graphs is $k$-\emph{planar graphs}, that is, graphs that can be drawn in the plane with at most $k$ crossings per edge. Various properties of $k$-\emph{planar graphs} are known: their maximum edge density, relation to other families of beyond-planar graphs, as well as the complexity of many algorithmic problems~\cite{bp-survey}.
The special case of $1$-\emph{planar graphs} has been extensively studied, for example, their maximum edge density is $4n-8$~\cite{schumacher-1-planar}, which is tight, and they can be colored with at most seven colors~\cite{ringel1965sechsfarbenproblem}. Even though this might seem to generalize planar graphs only slightly, in contrast to planar graphs, determining whether a given graph is $1$-planar is NP-hard~\cite{grigoriev-2007,korzhik-2008}. Moreover, their relation to other families of beyond-planar graphs is well understood~\cite{relation}.
\shahininline{We need to continue this paragraph with a summary of results on 1-planar graphs. E.g., cite the existing results on their coloring number and a few other results to highlight their siginificance. SD: Let's save this for the final version. MK: mentioned some more results (coloring, edge density, and relation to other beyond-planar graphs) about 1-planar graphs for motivation.}
In this paper we only consider \emph{simple} drawings, that is, edges can cross at most once and adjacent edges do not cross.

\paragraph{\textbf{Our results}}
The upper bound of Quilliot~\cite{q-asnapgpoagwagg-83} on the cop number of graphs with bounded genus was one of the first steps in the study of the Cops and Robbers game on beyond-planar graphs. 
In this work, we extend this direction by studying bounds on the cop number in another prominent family of beyond planar graphs: $k$-planar graphs. Graphs of bounded genus and $k$-planar graphs are unrelated; i.e.,  there are 1-planar graphs with arbitrary large genus, and graphs that have genus one but require arbitrarily many crossings per edge to be drawn in the plane~\cite{eppstein-math-overflow}. 
We show in Section~\ref{sec:1-planar-unbounded-cr} that, despite the fact that planar graphs have cop number at most three, 1-planar graphs may have unbounded cop number; see Theorem~\ref{thm:1PlanarUnbounded}.
Constructing 1-planar graphs with large cop number, as described in the proof of Theorem~\ref{thm:1PlanarUnbounded}, results in sparse 1-planar graphs with many vertices of degree two. 
With this in mind, we consider dense 1-planar graphs, particularly maximal 1-planar graphs, in Section~\ref{sec:max1Planar}. 
A 1-planar graph is said to be \emph{maximal} if no edge can be added such that the resulting graph remains 1-planar.
Maximal 1-planar graphs with $n$ vertices and the maximum number $4n-8$ of edges~\cite{scgumacher-struktur-1986} are called \emph{optimal 1-planar}.  
Maximal 1-planar graphs can have fewer than $4n-8$ edges, e.g., there exist maximal 1-planar graphs that have as few as $2.65n$ edges~\cite{brand-1-planar-density-2012}. 
%
In Section~\ref{sec:quad}, we construct a \emph{planar quadrangulation} (a plane graph in which each face has exactly four edges) with cop number three.
Using this quadrangulation, in Section~\ref{sec:max1Planar-lowerBound}, we show how to construct a maximal 1-planar graph (which is also an optimal 1-planar graph) for which three cops are necessary to catch a robber. 
In Section~\ref{sec:max1Planar-upperBound}, we show that each maximal 1-planar graph has cop number at most three using the Jordan arc property, as used by Aigner and Fromme~\cite{af-agocar-82}, and structural properties of maximal 1-planar graphs. 

In Section~\ref{sec:outer-k-planar}, we consider \emph{outer $k$-planar graphs}~\cite{auer2016,chaplick-bp-2018}, where, in addition to each edge being crossed at most $k$ times, we also restrict the vertices to be in convex position on the outer boundary.  The cop number of an outer $k$-planar graph is bounded by $1.5k+6.5$ due to its relation to treewidth~\cite{subdivision-lemma-and-tw-over-2} (if the treewidth of a graph is at most $t$, then its cop number is at most $t/2 + 1$~\cite{subdivision-lemma-and-tw-over-2}). 
This is an interesting contrast to general $k$-planar graphs, which can have unbounded cop number. 
Finally, we examine outer 1-planar graphs, which have cop number at most two due to the upper bound via their treewidth, and characterize those that are cop-win. 
\shahininline{I suggest removing the discussion on k-outerplanar graphs in the main paper and only mentioning that -along with the treewidth results- in the conclusion section and a topic for future work.}

\shahininline{The paper can benefit from a table with a summary of all results, including previous work (results on planar and bounded genus graphs) along with our contributions on I) (general) 1-planar graphs II) maximal outer-planar graphs  III) outer 1-planar graphs IV) the trivial upper bound for k-outerplanar graphs. We can link to the theorems in each case for better navigation of the paper. }

\section{$1$-Planar Graphs May Have Unbounded Cop Number}\label{sec:1-planar-unbounded-cr}
In this section we show that there exist 1-planar graphs with unbounded cop number.

\begin{theorem}
\label{thm:1PlanarUnbounded}
For every $c \in \mathbb{N}$, there exists a 1-planar graph with cop number at least $c$.
\end{theorem}
\begin{proof}
For any $c$, there exists a graph $G$ with cop number $c$ \cite[Theorem 2.6]{Neufeld1998}. Choose a $k$-planar drawing of $G$ for some $k$. Define a graph $G'$ that is obtained from $G$ by subdividing each edge $k-1$ times if $k$ is odd, or $k$ times if $k$ is even; see Figure~\ref{fig:subdivide}. 
Berarducci and Intrigila~\cite[Theorem 5.6]{berarducci-subdivision-lemma-1993} showed that replacing each edge with the same odd-length path does not decrease the cop number, so it follows that $c(G') \geq c(G)$. 
Also, $G'$ is 1-planar as we can obtain a 1-planar drawing of $G'$ from the $k$-planar drawing of $G$: between every two consecutive crossings on an edge in $G$, place one of the added subdivision vertices in $G'$.
\shahininline{Is it important to start from the drawing with min crossing? if not, we do not emphasize on that. Also, should'nt we subdivide each edge exactly k times? FL: we are trying to introduce as fewer vertices as possible, original we want $c(G)=\theta(\sqrt{n})$} \shahininline{I suggest adding a figure and moving the proof to in its own section rather than the contribution section.}
%
 \qed
\end{proof}

\begin{figure}
    \centering
    \includegraphics{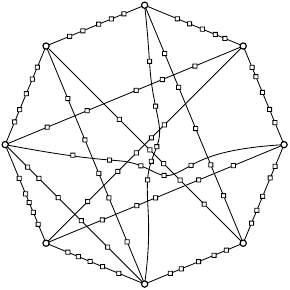}
    \caption{Subdividing each edge an equal number of times does not reduce the cop number of the resulting graph.}
    \label{fig:subdivide}
\end{figure}

\section{Cop Number of Quadrangulations}\label{sec:quad}


In this section, we study the cop number of \emph{quadrangulations}, that is, plane graphs in which every face has exactly four edges and four vertices. A graph $G$ can be drawn as a quadrangulation in the plane if and only if $G$ is a maximal planar bipartite graph~\cite{quad-mpb}. 
Since quadrangulations are planar, they have cop number at most three~\cite{af-agocar-82}. We show that there exists a quadrangulation~$Q$ for which three cops are necessary; see Theorem~\ref{thm:quadrangulation-cn-3}.  We will use this result in Section~\ref{sec:max1Planar} to prove 
that there are maximal 1-planar graphs with cop number three. 
To construct our quadrangulation $Q$ we first construct a  triangulation~$T$~\cite{workshop} which is based on the dodecahedron graph $D$; see Figure~\ref{fig:quadrangulation-cn-3}. The triangulation $T$ is constructed by adding vertices and edges to each face of $D$ in such a way that each resulting face is a triangle; see Figure~\ref{fig:dodecahedron-face}. The original dodecahedron edges are in black and the new triangulation edges are in dark green. We construct our quadrangulation $Q$ by adding one vertex and three edges to each face of this triangulation in such a way that each resulting face is a quadrangle; see Figure~\ref{fig:dodecahedron-face}. The new quadrangulation edges are in light green. 
After quadrangulating the triangulation, we have subdivided the original edges of the dodecahedron once, so  
the vertices of the dodecahedron that were initially adjacent are now at distance two from each other. Let $V'\subset V(Q)$ be the set of such vertices of $Q$ that subdivide the original edges of the dodecahedron.


\begin{theorem}
\label{thm:quadrangulation-cn-3}
The quadrangulation $Q$ has cop number three.
\end{theorem}
\begin{proof}
Because the quadrangulation $Q$ is a planar graph, its cop number is at most three~\cite{af-agocar-82}.
We will show that $Q$ has cop number at least three. In particular, we show that if there are only two cops, there are vertices in $V(D)\cup V' \subset V(Q)$ that the robber can choose to move to without being caught, using the following strategy. For any vertex $v$, denote by $N[v]$ the set of vertices within distance at most one from $v$, and denote by $N^2[v]$ the set of vertices within distance at most two from $v$. The robber begins on a vertex $r$ of the dodecahedron $D$, and remains on this vertex until a cop moves to a vertex adjacent to $r$ in $Q$. \todo[inline]{MK: this is not necessary, we can show an even stronger result, namely, that the robber can be active, that is, move every time on its turn}
For each neighbour $d$ of $r$ in the dodecahedron $D$, denote by $d' \in V'$ the vertex that subdivides the edge $rd$. \todo[inline]{MK: for notation consistency let's use simply $rd$ to denote an edge, previously it was $\{r,d\}$} The robber inspects the neighbors of $r$ in the dodecahedron $D$ and chooses such a neighbor $d$ that $N[d'] \cup N^2[d]$ contains no cops, and the robber's next two turns consist of moving to $d'$ then $d$. Once the robber arrives at the dodecahedron vertex $d$, it repeats the above strategy. 

We now prove by induction that the robber can follow the strategy indefinitely and will never be captured by any of the two cops. Suppose we're at the start of round $k$, the robber is at a vertex $r\in V(D)$, and it is not yet caught. This means that neither of the two cops is at $r$. If there are no cops adjacent to $r$, then the robber stays where it is, so it is not yet caught and it is located at a vertex in $V(D)$ at the start of round $k+1$. Otherwise, suppose there is at least one cop adjacent to $r$. Let the three neighbors of $r$ in $D$ be $d_1, d_2$, and $d_3$ and the corresponding dodecahedron subdivision vertices be $d'_1, d'_2$, and $d'_3$; see Figure~\ref{fig:quadrangulation--cn-3-proof}. For each $i, j \in \{ 1, 2, 3\}$ such that $i\neq j$, the only vertex shared by $N[d'_i] \cup N^{2}[d_i]$ and $N[d'_j] \cup N^{2}[d_j]$ (see the gray regions in Figure~\ref{fig:quadrangulation--cn-3-proof}) is $r$ which, by the induction assumption, is not occupied by any of the cops. Therefore, each cop can be in at most one vertex in $N[d'_i] \cup N^{2}[d_i]$, for~$i\in\{1, 2, 3\}$. Since there are only two cops, it follows that, for some $j \in \{1,2,3\}$, the set of vertices $N[d'_j] \cup N^{2}[d_j]$ is cop-free. Therefore, the robber can follow the given strategy, i.e., moving to $d_j'$ in round $k$ then $d_j$ in round $k+1$. No cop can capture the robber at $d_j'$ in round $k$ since there were no cops in $N[d'_j]$ at the start of round $k$, and, no cop can capture the robber at $d_j$ in round $k+1$ since there were no cops in $N^2[d_j]$ at the start of round $k$. Thus, the robber is not yet caught and it is located at a vertex in $V(D)$ at the start of round $k+2$. \shahin{Made some changes to clarify the proof. It still needs more attention.} \avery{I've made a pass through the proof, hopefully good now?}
\qed
\end{proof}

\shahininline{If (a) depicts $D$, add it to the figure. Same for (b) and Q. Also, show its triangulation that is used to form (b). Keep the scale of (a) in (b); it is a bit confusing now. (that is, (b) must be a larger graph).}
\begin{figure}
  \hfill
  \begin{subfigure}[b]{.45\textwidth}
    \centering
    \includegraphics[page=1]{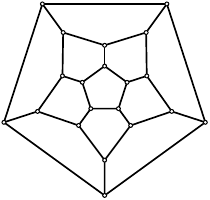}
    \caption{The dodecahedron graph $D$.}
    \label{fig:dodecahedron}
  \end{subfigure}
  \hfill
  \begin{subfigure}[b]{.45\textwidth}
    \centering
    \includegraphics[page=2]{dodecahedron}
    \caption{Each face of the dodecahedron graph $D$ in the quadrangulation $Q$.}
    \label{fig:dodecahedron-face}
  \end{subfigure}
  \caption{Construction of a quadrangulation $Q$ with cop number three from the dodecahedron graph $D$ (the skeleton of a dodecahedron).}
  \label{fig:quadrangulation-cn-3}
\end{figure}

\begin{figure}
    \centering
    \includegraphics{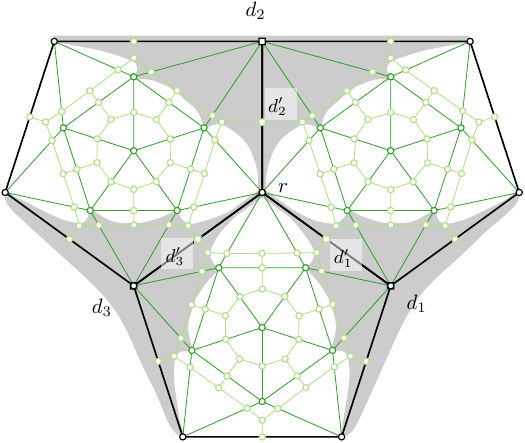}
    \caption{Three faces of the dodecahedron graph $D$ in quadrangulation $Q$. The edges of $D$ in black, of $T$ in dark green, and of $Q$ in light green.}
    \label{fig:quadrangulation--cn-3-proof}
\end{figure}
\shahininline{Fig.2: Make the vertices larger if possible. It is not clear that $d'_2$ is indeed a vertex. (MK: increased the vertex sizes) Btw, does the proof work for a single $Q$ or for three $Q$s as subgraphs? Is it clarified in the proof? MK: not sure if I understand the question correctly. We use only one copy of $Q$ in the proof}


\section{Cop Number of Maximal 1-Planar Graphs}\label{sec:max1Planar}

\subsection{Lower Bound}\label{sec:max1Planar-lowerBound}
We prove that there exists a maximal 1-planar graph, which is also optimal 1-planar, with cop number at least three. The structure of optimal 1-planar graphs is well known:

\averyinline{Are we actually citing anything from Brinkmann et al.? MK: good point,  Brinkmann et al. do not talk about crossings, removed the ref}
\begin{lemma}[Suzuki~\cite{s-reom1pg-2010}]
\label{lem:1-planar-subgraphs}
For any embedded graph $G$, denote by $Q(G)$ the subgraph of $G$ induced by the non-crossing edges. Then: (1) If $G$ is an embedded optimal 1-planar graph, then $Q(G)$ is a plane quadrangulation; and, (2) Let $H$ be an arbitrary simple planar quadrangulation. There exists a simple optimal 1-planar graph $G$ such that $H=Q(G)$ if and only if $H$ is $3$-connected.
 %
\end{lemma}
Observe that the quadrangulation $Q$ defined in Section~\ref{sec:quad} is 3-connected. Therefore, according to the proof of Lemma~\ref{lem:1-planar-subgraphs} in \cite{s-reom1pg-2010}, we can construct an optimal 1-planar graph $Q'$ from $Q$ by adding the crossing diagonal edges in each quadrangular face, as illustrated by the brown edges in Figure~\ref{fig:optimal-cn-3-proof}. 

\begin{figure}
    \centering
    \includegraphics{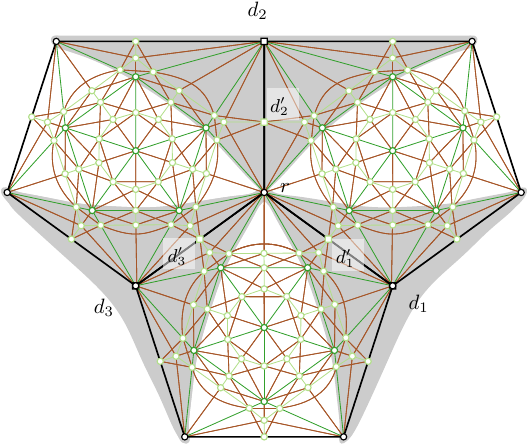}
    \caption{$Q'$ is formed from $Q$ by adding crossing diagonal edges (brown).}
    \label{fig:optimal-cn-3-proof}
\end{figure}

\shahininline{It may be better to change Theorem 2's statement to: There are maximal 1-planar graphs with cop number at least 3. Refer to $Q'$ in the proof and before stating theorem. MK: I see the point if this being a more general statement but we say this in the introduction and here we make a statement which is more informative, namely, states exactly which graph we mean}
\begin{theorem}
\label{thm:optimal-cn-3}
The maximal 1-planar graph $Q'$ has cop number at least three.
\end{theorem}
\begin{proof}
The proof uses the same robber strategy and induction proof as Theorem~\ref{thm:quadrangulation-cn-3}. The only difference is that the closed neighborhood $N[d'_i]$ of $d'_i$ and the closed distance-two neighborhood $N^{2}[d_i]$ of $d_i$, for each $i=1, 2, 3$, are larger, but it still holds that
the only vertex shared by $N[d'_i] \cup N^{2}[d_i]$ and $N[d'_j] \cup N^{2}[d_j]$ (see the gray regions in Figure~\ref{fig:optimal-cn-3-proof}) is $r$. Therefore, for some $j \in \{1,2,3\}$, the set of vertices $N[d'_j] \cup N^{2}[d_j]$ is cop-free, and the robber can move to vertex $d_j \in V(D)$ via $d'_j$ without being caught.
\shahininline{Even if you use the same induction, it may be good to re-state in a complete proof in the appendix (call this proof a sketch).}
\qed
\end{proof}

\subsection{Upper Bound}
\label{sec:max1Planar-upperBound}

Aigner and Fromme~\cite[Theorem 6]{af-agocar-82} proved that planar graphs have cop number at most three. One of the main arguments in their proof is the \emph{Jordan curve property}, i.e., that a Jordan curve partitions the surface into two connected regions: interior and exterior. Bonato and Mohar~\cite{bm-tdicar-17} noted that this approach could also be used for designing a three-cop strategy for other families of graphs that share this property, i.e., graphs with drawings in which there exist special separating cycles (that will play the role of the Jordan curves) that divide the drawing into interior and exterior.
Once we introduce crossings, however, it becomes unclear how to identify such separating cycles, and this approach fails for 1-planar graphs (see Theorem~\ref{thm:1PlanarUnbounded}). In this section, we show that we can still use the Jordan curve property to devise a three-cop strategy for maximal 1-planar graphs; see Theorem~\ref{thm:maximal-cn-3}. We will take advantage of the structure of maximal 1-planar graphs, in particular, the fact that
in every maximal 1-planar drawing, every pair of crossing edges $ab$ and $cd$ is enclosed in a quadrangle $acbd$ with four uncrossed edges (forming a \emph{kite} together with the pair of crossing edges)~\cite{barat-1-planar-density-2018}; see Figure~\ref{fig:kite}. Our proof strategy will follow closely that of Bonato and Nowakowski~\cite[Theorem 4.25]{cr-book}; however, we will need different building blocks that take into account the structure of maximal 1-planar graphs.

\begin{figure}
    \centering
    \includegraphics{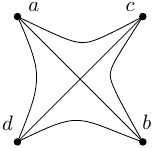}
    \caption{In a maximal 1-planar graph, every pair of crossing edges induces a kite.}
    \label{fig:kite}
\end{figure}

\begin{lemma}
\label{lem:shortes-path-no-crossing}
Let $H$ be an induced subgraph of a maximal 1-planar graph, let $u,v\in V(H)$ be distinct, and let
$P$ be
a shortest path between $u$ and $v$ in $H$. 
Then no two edges of $P$ cross.
\end{lemma}
\begin{proof}
Assume for the sake of contradiction that two edges $ab$ and $cd$ of $P$ cross. Then there is a kite $acbd$ and there is a shorter path that takes one of the side edges ($ac$, $cb$, $bd$, or $da$) of the kite instead of the two crossing edges.
\qed
\end{proof}

\begin{lemma}
\label{lem:isometric-path-crossings}
Let $H$ be an induced subgraph of a maximal 1-planar graph such that the robber can only move in $H$, let $u,v\in V(H)$ be distinct, and let
$P$ be
a shortest path between $u$ and $v$ in $H$. 
Then a single cop on $P$ can, after a finite number of moves: (a) prevent the robber from entering $P$, that is, the robber will be immediately caught if they move onto $P$; and, (b) catch the robber immediately after it traverses an edge that crosses an edge of $P$.
\end{lemma}
\begin{proof}
    The proof of (a) holds for general graphs and can be found in the book of Bonato and Nowakowski~\cite[Theorem 1.7]{cr-book}. 
    The idea behind the proof is as follows. Because the shortest path $P$ does not have shortcuts, after a finite number of moves, the cop can follow the ``projection" of the robber onto the path, and thus, always remain close to the vertex of the path where the robber might enter.

    To prove (b), suppose that after sufficiently many moves, one of the cops guards the path $P$ in the sense of~(a). Assume the robber is at one of the endpoints of some edge $ab$,  say $a$, crossing some edge $v_iv_{i+1}$ of the path $P$. Since $H$ is an induced subgraph of a maximal 1-planar graph, $ab$ together with  $v_iv_{i+1}$ form a kite $av_ibv_{i+1}$. Consider a cop strategy in which $P$ is guarded by one cop in the sense of~(a).
    That is, the cop moves in a way to catch the robber as soon as the robber enters $P$. So, after a finite number of moves, if the robber is at $a$, the cop will be at $v_i$ or $v_{i+1}$, in anticipation of the robber moving to any of these vertices (note that the edges between $a$ and $v_i$ and $v_{i+1}$ exist due to the kite structure). That is, when the robber moves from $a$ to $b$, the cop will be at either $v_i$ or $v_{i+1}$ and can move to $b$ to catch the robber in the next step.  \sk{I rewrote this part for clarity.} 
\qed
\end{proof}
In light of Lemma~\ref{lem:isometric-path-crossings}, we say that a cop \emph{guards}  the shortest path $P$ between two vertices (referred to also as an \emph{isometric path}) if after a finite number of moves, it can prevent the robber from entering or crossing~$P$. Observe that, according to Lemma~\ref{lem:shortes-path-no-crossing}, no two edges of $P$ cross.

\begin{lemma}
\label{lem:jordan+isometric}
Let $H$ be an induced subgraph of a maximal 1-planar graph such that the robber can only move in $H$, let $u,v\in V(H)$ be distinct, and let
$P_1$ and $P_2$ be two internally disjoint paths from $u$ to $v$ such that $P_1$ is isometric in $H$, $P_2$ is isometric in $H-(V(P_1)\setminus \{u, v\})$, and no edge of $P_1$ crosses an edge of $P_2$. Then $P_1 \cup P_2$ can be guarded by two cops in~$H$.
\end{lemma}
\begin{proof}
According to Lemma~\ref{lem:isometric-path-crossings}, since $P_1$ is isometric in $H$, it can be guarded by one cop.  Similarly, since $P_2$ is isometric in $H-(V(P_1)\setminus \{u, v\})$ it can also be guarded by one cop.\qed
\end{proof}




\begin{restatable}{theorem}{onePlanarTheorem}
\label{thm:maximal-cn-3}
Any maximal 1-planar graph $G$ has cop number at most three. 
\end{restatable}

\begin{proof} 
We provide a sketch of the proof here; see the complete proof in Appendix~\ref{sec:app:thm:maximal-cn-3}.

Given that $G$ is a maximal 1-planar graph, it is 2-connected and thus must contain two internally disjoint shortest paths $P_1$ and $P_2$ between two vertices $v,w\in G$, where 
$P_1$ is isometric in $G$, $P_2$ is isometric in $G \setminus (V(P_1)\setminus \{u, v\})$, and no edge of $P_1$ crosses an edge of $P_2$. 
We let two cops $C_1$ and $C_2$ guard the two paths $P_1$ and $P_2$, respectively, as described in Lemma~\ref{lem:jordan+isometric}. 
Therefore, after a finite number of $t$ steps, the robber must avoid edges that are located on these paths and must avoid crossing them. 
%
%
%
Define the \emph{robber territory} $H$ as the subgraph induced by the vertices in the interior or exterior of the cycle $P_1 \cup P_2$, whichever contains the robber at step $t$. 
Given that the robber territory is non-empty and $G$ is a maximal 1-planar graph, one can consider a third isometric path $P_3$ between $v$ and $w$ that contains a vertex in the robber territory; this path can share some vertices with $P_1$ or $P_2$ (or both). We let the third cop $C_3$ guard $P_3$, as in Lemma~\ref{lem:isometric-path-crossings}. Therefore, after a finite number of steps, the robber's ``safe zone" becomes limited to one of the parts of the partition of the robber's territory created by $P_3$. We show that two cops can guard this area so that the robber stays confined to it, and the third cop can be used to repeat this process in the smaller subgraph that forms the safe zone of the robber.\qed
%
%
\end{proof}

\section{Characterization of Cop-Win Outer 1-Planar Graphs}
\label{sec:outer-k-planar}

A graph $G$ is \emph{outer $k$-planar} if it can be drawn in the plane so that all the vertices are in convex position on the outer boundary and each edge is crossed at most $k$ times. In this section, we provide a complete characterization of the cop number of outer 1-planar graphs.

First, we briefly discuss the cop number of general outer $k$-planar graphs. Joret et al.~\cite{subdivision-lemma-and-tw-over-2} showed that the cop number of a graph can be bounded in terms of its treewidth. In particular, if the treewidth of a graph is at most $t$, then its cop number is at most $t/2 + 1$~\cite{subdivision-lemma-and-tw-over-2}.
Every outer $k$-planar graph has treewidth at most $3k+11$~\cite[Proposition 8.5]{outer-k-planar-tw}, so, its cop number is at most $1.5k+6.5$. 

Restricting attention to outer 1-planar graphs, Auer et al. \cite{auer2016} showed that every outer 1-planar graph has treewidth at most three, which implies that the cop number of outer 1-planar graphs is at most two. Bonato et al.~\cite{Bonato2019OptimizingTT} showed that an outerplanar graph $G$ is cop-win if and only if $G$ is chordal. 
We generalize this result to outer 1-planar graphs, which completes our characterization of $c$-cop-win outer 1-planar graphs for all $c \in \{1, 2\}$. 


 Let $V = (v_1, \ldots, v_n)$ be the cyclic ordering of vertices of an outer 1-planar graph $G = (V, E)$. Let $V[v_i, v_j]$, be the set of vertices $(v_i, v_{i+1}, \ldots, v_j)$, for $i \le j$, and $V(v_i, v_j)$ be the set of vertices $(v_{i+1}, v_{i+2}, \ldots, v_{j-1})$, for $i < j$.

\begin{proposition}\label{prop:o1pcrossingedges}
    Let \(G = (V, E)\) be an outer 1-planar graph. Consider $U \subset V$ such that $G[U]$ forms a cycle. For every pair of vertices $u, w \in U$ such that $V[u, w] \cap U = \{u, w\}$, it holds that if $u$ and $w$ are not adjacent, then there are two edges $e_u, e_w \in E(U)$ incident to $u$ and $w$, respectively, such that $e_u$ and $e_w$ cross.
\end{proposition}
\begin{proof}
    Consider such a pair $u, w \in V$, and assume that neither of the cycle edges $uu_1$ nor $uu_2$ crosses any of the cycle edges $ww_1$ and $ww_2$; see Figure~\ref{fig:o1pcrossingedges}.
    Consider the path $P_1 = (w_1, \ldots, u)$ in $G[U]$ from $w_1$ to $u$ such that $w \notin P_1$. It follows that $w_2 \notin P_1$, and so there must be an edge $e_1 = w_1'w_1'' \in E(P_1)$ such that $w_1' \in V[w_1, w_2]$ and $w_1'' \in V[w_2, u_1]$ (note that the edge $e_1$ crosses $ww_2$). Next, consider the path $P_2 = (w_2, \ldots, u)$ in $G[U]$ from $w_2$ to $u$ such that $w, w_1 \notin P_2$. There must be an edge $e_2 = w_2'w_2'' \in E(P_2)$ such that $w_2' \in V[w_1', w_1'']$ and $w_2'' \in V[w_1'', u_1]$ (note that the edge $e_2$ crosses $e_1$). Therefore, $e_1$ is crossed by two edges, $ww_2$ and $w_2'w_2''$, contradicting the fact that $G$ is outer 1-planar. \qed
\end{proof}

\begin{figure}
    \centering
    \includegraphics{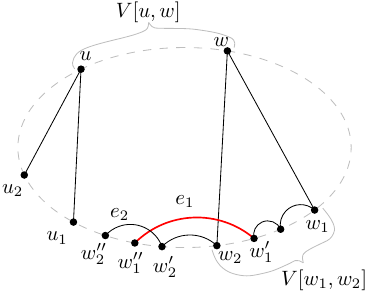}
    \caption{Illustration in support of Proposition~\ref{prop:o1pcrossingedges}.}
    \label{fig:o1pcrossingedges}
\end{figure}

\begin{corollary}\label{prop:o1pComponents}
    Let \(G = (V, E)\) be an outer 1-planar graph.  Consider $U \subset V$ such that $G[U]$ forms a cycle of order $k \ge 4$. For every $u, w \in U$ such that $V(u, w) \cap U = \varnothing$, the number of vertices in $V \setminus V(u,w)$ that are adjacent to at least one vertex in $V(u, w)$ is at most three. 
\end{corollary}

\begin{proof}
    By Proposition~\ref{prop:o1pcrossingedges}, either $u$ and $w$ are adjacent, or, there are two edges $e_u$ and $e_w$ incident to $u$ and $w$, respectively, that cross. First, suppose that $u$ and $w$ are adjacent. Then, each edge $vv' \in E(G)$ such that $v \in {V}(u, w)$ and $v' \in V \setminus {V}(u, w)$ and $v' \not\in \{u,w\}$ must cross the edge $uw$, so there can be at most one such edge $vv'$ since $G$ is outer 1-planar. Next, suppose that $u$ and $w$ are not adjacent. Consequently, for any $v \in {V}(u, w)$ and any $v' \in V \setminus {V}(u, w)$ with $v' \not\in \{u,w\}$, an edge $vv'$ would have to cross either $e_u$ or $e_w$. But since $G$ is outer 1-planar and there is a crossing between $e_u$ and $e_w$, it follows that no such edge $vv'$ exists. So, in both cases above, we proved that there is at most one vertex $v' \in V \setminus {V}(u, w)$ with $v' \not\in \{u,w\}$ such that $v'$ is adjacent to at least one vertex in ${V}(u, w)$. The fact that $u$ and/or $w$ might be adjacent to at least one vertex in ${V}(u, w)$ gives the desired upper bound of three.\qed
\end{proof}

\begin{theorem}
\label{thm:outer-1-planar-characterization}
A connected outer 1-planar graph $G = (V,E)$ is cop-win if and only if it is chordal.
\end{theorem}

\begin{proof}
It is known that if $G$ is chordal, then it is cop-win (in fact, this is true for any graph~\cite{quilliot-french-1978}). Thus it remains to show that, if $G$ is not chordal, then there is a robber strategy to escape one cop.

Suppose that $G$ is not chordal. Then there is a set $U \subset V$ so that $|U| \geq 4$ and $G[U]$ is a cycle. Let $v_1, \dots, v_n$ be the cyclic ordering of $G$. For all $i \leq j$, let $V(v_i, v_j)$ be the set of vertices $(v_{i+1},v_{i+2},\dots,v_{j-1})$. The robber's strategy is to maintain the following invariants at the end of each of its turns, and will only move if one of them is violated at the start of its turn:
\begin{enumerate}
    \item If the cop is on a vertex $u \in U$, then the robber is on a vertex in $U \setminus N[u]$, where $N[u]$ is the closed neighborhood of $u$.
    \item If, for some $u,w \in U$ such that $V(u,w) \cap U = \varnothing$, the cop is on a vertex in $V(u,w)$, then the robber is on a vertex in $U$ that is not adjacent to any vertex in $V(u,w)$.
\end{enumerate}

Now, we show that this is always possible. We begin by considering the initial position of the cop and robber. First, if the cop is on a vertex $u \in U$, we have that $|N[u] \cap U| = 3$. Then since $|U| \geq 4$, the robber can choose an initial vertex in $U \setminus N[u]$. Second, if the cop is on a vertex in $V(u,w)$ for some $u,w \in U$ such that $V(u,w) \cap U = \varnothing$, then, since there are at most three vertices in $V \setminus V(u,w)$ that are adjacent to at least one vertex in $V(u,w)$ by Corollary~\ref{prop:o1pComponents}, and $|U| \geq 4$, the robber can choose an initial vertex in $U$ that is not adjacent to any vertex in $V(u,w)$. Now we must show that the robber can maintain these invariants. Let $r \geq 1$ and suppose that the invariants are true at the end of the robber's turn in the $r$'th round. This implies that the cop and robber are not adjacent, so the cop cannot catch the robber on its $(r+1)$'st turn. If neither invariant is violated at the beginning of the robber's $(r+1)$'st turn, then we are done, so suppose that one of the invariants is violated. Suppose that the first invariant is violated, i.e., the cop is on a vertex $u_1 \in U$ and the robber is on a vertex $u_2 \in U \cap N(u_1)$. Then, since $|U| \geq 4$, there must be a vertex $u \in N(u_2) \setminus N(u_1)$, so the robber can move to $u$ and the first invariant is satisfied at the end of the robber's turn. Next, suppose that the second invariant is violated, i.e., the cop is on a vertex in $V(u,w)$ for some $u,w \in U$ such that $V(u,w) \cap U = \varnothing$, and the robber is on a vertex $v \in U$ that is adjacent to at least one vertex in $V(u,w)$. We will show that there is a vertex $v' \in U \cap N(v)$ that is not adjacent to a vertex in $V(u,w)$. Suppose, for the sake of a contradiction, that there is no such vertex $v'$. Let $\{v_1,v_2\} = N(v) \cap U$. Then $v, v_1, v_2$ are all adjacent to vertices in $V(u,w)$ by assumption. Now, since neither of the invariants were violated at the end of the robber's $r$'th turn, then we know that at the start of the cop's $(r+1)$'st turn, the cop was not on $v_1$, $v$, or $v_2$, nor was it in $V(u,w)$. But, by assumption, after the cop's $(r+1)$'st turn, the cop is on a vertex in $V(u,w)$. Thus, at the start of the cop's $(r+1)$'st turn, it must be the case that the cop was on a vertex in $V \setminus \{v_1,v,v_2\}$ that is adjacent to at least one vertex in $V(u,w)$. However, this implies that there are at least four vertices in $V \setminus V(u,w)$ that are adjacent to at least one vertex in $V(u, w)$, which contradicts Corollary~\ref{prop:o1pComponents}. By reaching a contradiction, we have proved that there exists a vertex $v' \in U \cap N(v)$ that is not adjacent to a vertex in $V(u,w)$. Thus, the robber can move to this $v'$ and the second invariant is satisfied at the end of the robber's turn.\qed
\end{proof}

\section{Discussion and Directions for Future Research}\label{sec:discussion}

%
%

The cop number of an outer $k$-planar is at most $3k/2 + 13/2$ due to the fact that its treewidth is bounded by $3k+11$ and the relation between the cop number and treewidth of a graph~\cite{subdivision-lemma-and-tw-over-2}. Two questions follow naturally. Is there a non-trivial lower bound expressed as a function of $k$? Can we reduce the multiplicative factor of $3/2$ in the upper bound?


Determining whether a graph $G$ is $k$-cop-win is NP-hard, and $W[2]$-hard parameterized by $k$~\cite{fomin-cr}. Very recently Gahlawat and Zehavi~\cite{fpt} showed that the problem is fixed parameter tractable (FPT) in vertex cover. We would like to restate their open question: whether the problem is FPT in feedback vertex set, treewidth, or treedepth.

\bibliographystyle{splncs04}
\bibliography{refs}


\appendix

\section{Proof of Theorem~\ref{thm:maximal-cn-3}}
\label{sec:app:thm:maximal-cn-3}

Let us define some additional useful notation. 
Let $X$ be a cycle without edge crossings in a maximal 1-planar graph $G$, and let $u$ be a vertex of $G-X$.
The cycle $X$ partitions the plane into two regions: $A_1$ that contains $u$
and $A_2$ which does not. Let $V_1$ denote the vertices contained in $A_1$,
called the interior of $X$ with respect to $u$, and $V_2$ those in $A_2$ which
we call the exterior of $X$ with respect to $u$. The subgraph induced
by $V(X)\cup V_1$ is called \emph{the internal subgraph} determined by $X$, and
written $\intr{X}$. The subgraph induced by $V(X) \cup V_2$ is called \emph{the
external subgraph} determined by $X$, and written $\extr{X}$. We note that
the only way for the robber to pass from the interior to the exterior (or
vice versa), is to pass through a vertex of the cycle $X$ or cross an edge of $X$.

\onePlanarTheorem*
\begin{proof}
Let $G'$ be an induced subgraph of the maximal 1-planar graph $G$. Consider the following steps.

\noindent \textbf{Inductive argument:}
$G'$ is 2-connected and thus must contain two internally disjoint shortest paths $P_1$ and $P_2$ between two vertices $v,w\in G'$, where 
$P_1$ is isometric in $G'$, $P_2$ is isometric in $G' \setminus (V(P_1)\setminus \{u, v\})$, and no edge of $P_1$ crosses an edge of $P_2$. 
We let two cops $C_1$ and $C_2$ guard the two paths $P_1$ and $P_2$, respectively, as described in Lemma~\ref{lem:jordan+isometric}. In the light of Lemma~\ref{lem:isometric-path-crossings}, after a finite number of $t$ steps, the robber must avoid edges that are located on these paths and must avoid crossing them.


Define the \emph{robber territory} $H$ as the subgraph induced by the vertices in the interior or exterior of the cycle $P_1 \cup P_2$, whichever contains the robber. Define the \emph{cop territory} as $G' \setminus H$. Note that any path from the robber to the cop territory $T$ crosses 
$P_1 \cup P_2$ (either at a vertex or via an edge crossing). Therefore, 
the robber must choose to move to or stay at a vertex in $H$ if it wants to avoid capture.

\todo{Do we have to show the induced subgraph is always 2-connected? MK: good point, I think, it follows from how we choose $G'$}
\shahin{I added a note that ``H is a 1-maximal planar graph that is smaller than G'"}
\todo{The induced subgraph H is not maximal 1-planar graph anymore MK: that's true, but we only need $G'$ ($G'$ is $H$ together with $P_1$ and $P_2$) to be a 2-connected subgraph of a maximal 1-planar graph}

We show by induction on the size of $H$ that in a finite number of steps, the three cops $C_1, C_2$, and $C_3$ can move so that the robber territory $H$ decreases by at least one vertex, and we are again in the situation as described in our inductive assumption with $G' = H$. Therefore, by the induction hypothesis, the three cops can catch the robber, given that $H$ has smaller size than $G'$. 

\textbf{Base case:} 
When $H$ is formed by only one vertex $v$, the robber will be trivially caught by the cops in a finite number of rounds; if it crosses $P_1$ or $P_2$, it will be caught by the above observation, and if it stays in $v$, the cops eventually approach the robber and catch it at $v$.

\textbf{Inductive step:}
Now suppose the induction assumption holds. Without loss of generality, we
may assume the robber is located in $\intr{X}$, where $X = P_1 \cup P_2$, and
the subgraph $\extr{X}$ is guarded by $C_1$ and $C_2$. 
Now $G'$ is the subgraph induced on $V(\intr{X})$ and $H$ is the subgraph induced on $V(\intr{X}) \setminus V(X)$. 
%
Let $P_3$ be a simple shortest path between $v$ and $w$ in $G'$ that has at least one vertex of $H$. Given that $H$ is non-empty and $G'$ is 2-connected, such path $P_3$ must exist. \avery{Fengyi pointed out to me: the proof that $P_3$ exists is not rigorous enough (and Shahin also commented on this on Tuesday?)}
According to Lemma~\ref{lem:shortes-path-no-crossing}, $P_3$ does not self-cross (neither at a vertex nor via an edge crossing).
Therefore, the path $P_3$ partitions $G'$ into induced subgraphs $G'_i$ (sharing only the vertices of $P_3$), for $i=1, \dots, \ell$, that are bounded by a contiguous part of $P_3$ and a 
contiguous part of either $P_1$ or $P_2$; see Figure~\ref{fig:regions}. After a finite number of moves, when $C_3$
can guard $P_3$, the robber is trapped in one of such subgraphs, say $G'_i$, as we will describe next. 

Assume $G'_i$ is bounded by $P_1$ and $P_3$, and $P_3$ leaves $P_1$ at some vertex 
$v_1$ (possibly $v_1= v$) and enters back at some vertex $w_1$ (possibly $w_1= w$). Let $P'_1$ be the part of $P_1$ between $v_1$ and $w_1$ and $P'_2$ be the part of $P_3$ between $v_1$ and $w_1$. Observe that $P'_1$ is a shortest path between $v_1$ and $w_1$ in $G'_i$ and $P'_2$ is a shortest path between $v_1$ and $w_1$ that is vertex disjoint from $P'_1$. Therefore, we are back to our induction assumption with $G' = G'_i$, $P_1 = P'_1$, $P_2 = P'_2$, and $H = G'_i - P'_1 - P'_2$. By the induction hypothesis, the three cops will eventually catch the robber, given that $H$ is a maximal 1-planar graph that is smaller than $G'$.

Assume now that $G'_i$ is bounded by $P_2$ and $P_3$, and $P_3$ leaves $P_2$ at some vertex 
$v_2$ (possibly $v_2= v$) and enters back at some vertex $w_2$ (possibly $w_2=w$). Let $P'_1$ be the part of $P_3$ between $v_2$ and $w_2$ and $P'_2$ be the part of $P_2$ between $v_2$ and $w_2$. The path $P'_1$ is a shortest $v_2w_2$-path in $G'_i$; otherwise, $P_3$ would have a shortcut. Therefore, we are again back to our induction assumption with $G' = G'_i$, $P_1 = P'_1$,  $P_2 = P'_2$, and $H = G'_i - P'_1 - P'_2$. Again, by the induction hypothesis, the three cops will eventually catch the robber, given that $H$ is a maximal 1-planar graph that is smaller than $G'$.\qed
\end{proof}
\begin{figure}
    \centering
    \includegraphics{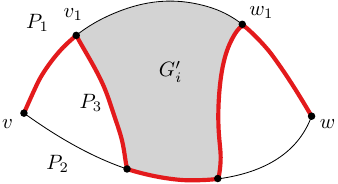}
    \caption{The path $P_3$ partitions $G'$ into induced subgraphs $G'_i$ (sharing only the vertices of $P_3$).}
    \label{fig:regions}
\end{figure}

\end{document}